\documentclass[12pt]{amsart}
\usepackage{epsfig,color}
\usepackage{amssymb} 

\makeatother

\theoremstyle{definition}

\def\fnum{equation}
\newtheorem{Thm}[\fnum]{Theorem}
\newtheorem{Cor}[\fnum]{Corollary}

\newtheorem{Lem}[\fnum]{Lemma}

\newtheorem{Pro}[\fnum]{Proposition}

\numberwithin{equation}{section}

\newcommand{\Vol}{{\text{Vol}}}

\newcommand{\LL}{{\mathcal{L}}}

\newcommand{\Ric}{{\text{Ric}}}

\newcommand{\Hess}{{\text {Hess}}}

\def\RR{{\bold R}}

\def\LL{{\bold L }}
\newcommand{\dv}{{\text {div}}}
\newcommand{\e}{{\text {e}}}

\newcommand{\cL}{{\mathcal{L}}}

\newcommand{\R}{{\text{R}}}

\newcommand{\cP}{{\mathcal{P}}}

\newcommand{\eqr}[1]{(\ref{#1})}

\title{Propagation of symmetries for Ricci shrinkers}
\author[]{Tobias Holck Colding}%
\address{MIT, Dept. of Math.\\
77 Massachusetts Avenue, Cambridge, MA 02139-4307.}
\author[]{William P. Minicozzi II}%

\thanks{The  authors
were partially supported by NSF  DMS Grants 2104349 and 2005345.}

\email{colding@math.mit.edu  and minicozz@math.mit.edu}

\begin{document}

\maketitle

{\centering\footnotesize Dedicated to our friend David Jerison.\par}

\begin{abstract}
We will show that if a gradient shrinking Ricci soliton has an
  approximate symmetry on one scale, this   symmetry  propagates   to larger scales.  This is an example of the shrinker principle  which roughly states that 
information radiates outwards for shrinking solitons.
\end{abstract}

 \section{Introduction}

 A time-varying metric $\bf{g}$ on a manifold $M$ is a Ricci flow if
 \begin{align}
 	\partial_t \, \bf{g} = - 2 \, \Ric_{\bf{g}} \, .
 \end{align}
 This is a nonlinear geometric evolution equation where singularities can form and understanding them is the key for understanding the flow.     
 For instance, one important ingredient in the rigidity of cylinders for the Ricci flow in \cite{CM4} was a new estimate called propagation of almost splitting.  This showed that
 if a gradient shrinking Ricci soliton almost splits off a line on one scale, then it also almost splits on a strictly larger set, though with a loss on the estimates.  
 
 A manifold splits off a line if it is the metric product of a Euclidean factor $\RR$ with a manifold of one dimension less. A splitting  gives a linear function whose gradient is
 a parallel vector field.
 Similarly, 
 an almost splitting gives an almost parallel vector field, which is a special type of almost Killing vector field.  Thus, the propagation of almost splitting is equivalent to showing that 
 a certain type of approximate symmetry extends to larger scales.  We will see here that this holds also for more general approximate symmetries.
 
 A triple $(M,g,f)$ of a manifold $M$, metric $g$ and function $f$  
is a gradient  Ricci soliton   if there is a constant $\kappa$ so that
\begin{align}	\label{e:kappahere}
\Ric + \Hess_f = \kappa \, g \, .
\end{align}
Up to diffeomorphisms, the Ricci flow of  a gradient Ricci soliton evolves by shrinking when $\kappa > 0$, is static (or steady) when $\kappa = 0$, and expanding
when $\kappa < 0$.  
Gradient shrinking Ricci solitons model finite time singularities of the flow and describe the asymptotic structure at minus infinity for ancient flows.

A Killing field $V$ is a 
  vector field   that generates an isometry, i.e.,     the Lie derivative  of the metric $g$ vanishes 
$$\LL_V g = 0 \, .$$
Equivalently, $V$ is Killing when its covariant derivative
 $\nabla V$ is skew-symmetric.   
  Many of the most important examples of solitons are highly symmetric; cf. \cite{BaK}, \cite{Br1}, \cite{Br2}, \cite{BrC}, 
 \cite{CIM}, \cite{CM4}, \cite{K}, \cite{LW}.

On Euclidean space, there are two types of Killing fields: the parallel vector fields that generate translations and the linearly growing rotation vector fields.
There are two very different types of symmetries on Ricci solitons, depending on whether or not the symmetry preserves the function $f$.   This dichotomy is already seen in the simplest examples of shrinkers: the Gaussian soliton $\left(\RR^n , \delta_{ij} ,  \frac{|x|^2}{4} \right)$ and cylinders.  Rotations around the origin on the Gaussian soliton  preserve both the metric and $f= \frac{|x|^2}{4}$, while the translations along the axis of a cylinder preserve the metric but not the level sets of $f$. 
 It is not hard to see that this is typical.   To make this precise, use the metric $g$ and function $f$, to define a weighted $L^2$ norm on functions, vector fields, and general tensors by 
\begin{align}
	\| \cdot \|_{L^2}^2 = \int_M |\cdot|^2 \, \e^{-f} \, .
\end{align}

\begin{Pro}	\label{t:RSK}
If $Y$ is an $L^2$ Killing field on a gradient shrinking Ricci soliton, then either $Y$ preserves $f$  or $M$ splits off a line.
\end{Pro}

The metric and weight induce  weighted divergence operators $\dv_f$ on vector fields and symmetric two-tensors and a drift Laplacian $\cL$ on general tensors.  
The adjoint $\dv_f^*$ of $\dv_f$ 
takes a vector field $V$ to the symmetric two-tensor 
\begin{align}
	\dv_f^* \, V = - \frac{1}{2} \, \LL_V \, g \, .
\end{align}
As in \cite{CM4}, define a self-adjoint operator $\cP$ on vector fields by
\begin{align}
	\cP \, Y = \dv_f \circ \dv_f^* \, Y \, .
\end{align}
  For a  Killing field $V$, $\dv_f^* \, V$ vanishes and, thus,
   $\cP \, V = 0$.

 For an  {\it{approximate}} Killing field,   the associated flow {\it almost} preserves the metric $g$.
As mentioned, approximate translations played an important role in \cite{CM4} to show   ``propagation of almost splitting''.  
Our interest here will be a corresponding propagation for more general symmetries.  

We will show that if a gradient shrinking Ricci soliton {{\it{(shrinker)}} has an
  approximate symmetry on one scale, then this approximate symmetry  propagates outwards to larger scales.  This is an example of the shrinker principle, \cite{CIM}, \cite{CM4}, which roughly states that 
information radiates outwards for these types of equations.

\begin{Thm}	\label{t:main}
Let $(M,g,f)$ be a non-compact shrinker and $C_1 $ a constant with $|\R| \leq C_1$.  There exist  $C_2 , R  $ so that if
$Y$ is a vector field on $\{ f < r^2/4 \}$ with $r \geq R$ and
\begin{enumerate}
\item $\int_{ f< r^2/4} |Y|^2 \, \e^{-f} = 1$ and $\int_{ f< r^2/4} |\dv_f^* Y|^2 \, \e^{-f} \leq \bar{\mu} \leq \frac{1}{4}$,
\item $  |Y| + |\nabla Y| \leq C_1 \, r$ on $\{ f < r^2/4  \}$,
\end{enumerate}
then there is a vector field $Z$ on $M$ with $  \|Z\|_{L^2} = 1$, $\cP \, Z = \mu \, Z$ and satisfying
\begin{enumerate}
\item[(Z1)]  $\| \dv_f^* Z \|_{L^2}^2 = \mu \leq C_2 \, \left( \bar{\mu} + r^{4+n} \, \e^{ - \frac{r^2}{4}} \right)$.
\item[(Z2)]  For $s \in (r^2/4 , r^2)$, we have $\int_{f=s} |\dv_f^* Z|^2 \leq C_2 \, r^{C_2} \, \mu$.
\end{enumerate}
\end{Thm}

The theorem shows that the approximate symmetry extends to the larger set, though with a loss in the estimates.   This extension
 is   powerful in situations where the loss can be recovered using some additional structure particular to the situation, thus leading to a global symmetry.   
 In Theorem $0.2$ in  \cite{CM4}, the loss was recovered by solving a gauge problem and then using a 
 rigidity property of cylinders that showed up at the quadratic level.   This quadratic rigidity  relied on the structure of the compact factor of the cylinder 
 and not just on the existence of the approximate translation. 
 
 This kind of propagation of symmetry often plays an important role in understanding the structure of solutions, as well as the rate of convergence of a Ricci flow to a singularity.

\section{Weighted manifolds and Killing fields}

We will be most interested in gradient shrinking Ricci solitons, but many of the results hold more generally.  
Killing fields preserve the metric, so they also preserve volume and, thus, are always divergence-free.  However, they do not necessarily preserve the weighted volume,
 so the weighted divergence $\dv_f \, Y$ of a Killing field $Y$ need not vanish.  We will see, however, that it does satisfy an eigenvalue equation on any gradient Ricci soliton (see equation
 \eqr{e:killingdd}
 below).  
 
 To show this, we recall some general   formulas
from   Lemma $2.8$  in \cite{CM4}  
 on any gradient Ricci soliton 
 (i.e., $(M,g,f)$ satisfies \eqr{e:kappahere} for some constant $\kappa$)
 for any vector field $Y$:
\begin{align}
\left( \cL + \kappa \right) \, \dv_f \, Y &= - \dv_f \, (\cP\,Y)\, ,\label{e:claim1}\\
\cL\,\nabla\,\dv_f\,(Y)&=-\nabla\,\dv_f\,(\cP\,Y)\, . \label{e:claim2}
\end{align}
If $Y$ is a Killing field on a gradient Ricci soliton, then  \eqr{e:claim1} gives that
 \begin{align}	\label{e:killingdd}
 	(\cL + \kappa) \, \dv_f Y = 0 \, .
\end{align}
As a consequence of this,   $\dv_f Y$ is harmonic:

 \begin{Cor}	\label{c:killharmonic}
If $Y$ is a Killing field on a gradient Ricci soliton, then $\dv_f \, Y$ is harmonic.
\end{Cor}

\begin{proof}
By \eqr{e:killingdd}, $(\cL + \kappa) \, \dv_f \, Y = 0$.  Since $\dv \,Y = 0$ for any Killing field, it follows that
\begin{align}
	\Delta \, \dv_f \, Y &= \cL \, \dv_f \, Y + \langle \nabla f , \nabla \dv_f \, Y \rangle \notag \\
	&= - \kappa \, \dv_f \, Y - \langle \nabla f , \nabla \langle   Y, \nabla f \rangle  \rangle  \\
	&= \kappa \, \langle \nabla f , Y \rangle - \langle \nabla_{\nabla f} Y , \nabla f \rangle - \Hess_f (Y , \nabla f) \, . \notag
\end{align}
The second to last term vanishes because of the Killing equation.  For the last term, we bring in the soliton equation $\Ric + \Hess_f = \kappa \, g$ to get
\begin{align}
	\Delta \, \dv_f \, Y = \Ric (\nabla f , Y ) \, .
\end{align}
Finally, this last term vanishes since $\Ric (\nabla f , Y) = \frac{1}{2} \, \langle \nabla S , Y \rangle$ (see, e.g., ($1.14$) in \cite{CM4}) and the scalar curvature must be constant along a Killing field.
\end{proof}

Later, will need  a second order self-adjoint operator $L$ defined on symmetric two-tensors by
\begin{align}
	L \, h = \cL \, h + 2 \, \R (h) \, ,
\end{align}
where $\R$ is the Riemann curvature acting on $h$ in an orthonormal frame by 
\begin{align}
	[\R (h)]_{ij} = \sum_{m,n} \R_{imjn} \, h_{mn} \, .
\end{align}
  Theorem $1.32$ in \cite{CM4} gives the following relation between $L$ and $\dv_f^*$
\begin{align}
	L\, \dv_f^*\,(Y)&= \dv_f^* \, (\cL + \kappa) \, Y \, .
	\label{e:firstc}
\end{align}
It will also be useful that the operator $\cP$ is related to $\cL$ by
\begin{align}
	-2\, \cP &= \nabla \, \dv_f + \cL + \kappa  \, .  \label{l:bochner} 
\end{align}
Finally, an easy integration by parts argument shows that
 if $Y$ and $\cP \, Y$ are in $L^2$, then $Y \in W^{1,2}$ and $\dv_f Y \in L^2$; this is given by the next lemma:

 \begin{Lem}	\label{l:interpcP}
 (Lemma $2.15$, \cite{CM4}) 
For any gradient Ricci soliton if $Y$, $\cP\,Y\in L^2$, then $\dv_f\,(Y)$, $\nabla\, Y\in L^2$ and
\begin{align}
	\| \nabla Y \|_{L^2}^2 + \| \dv_f \, Y \|_{L^2}^2 \leq 2 \, \| Y \|_{L^2} \, \| (2\, \cP + \kappa) \, Y \|_{L^2} \, .
\end{align}
\end{Lem}

 \subsection{Shrinkers}
 
 We  now specialize to shrinkers where   $\kappa = \frac{1}{2}$.

\begin{proof}[Proof of Proposition \ref{t:RSK}]
We consider two cases.  Suppose first that $Y$ preserves $f$, so that $\langle \nabla f , Y \rangle = 0$.
Since $Y$ is Killing, $\dv \, Y = 0$ and, thus,  $\dv_f \, Y$ vanishes.

Suppose now that $Y$ does not preserve  $f$ and, thus,  $\dv_f \, Y$ does not vanish identically.  In this case,   \eqr{e:killingdd} gives that
$\dv_f \, Y$ is a non-trivial   solution to
\begin{align}
	\cL \, \dv_f Y = - \frac{1}{2} \, \dv_f Y \, .
\end{align}
  Moreover, Lemma \ref{l:interpcP} implies that $\dv_f \, Y \in L^2$.   If $\cL v = -\mu \, v$, then
 the drift Bochner formula gives that 
 \begin{align}
 	\frac{1}{2} \, \cL \, |\nabla v|^2 = 
 |\Hess_{v}|^2 + \left( \frac{1}{2} - \mu \right) \, |\nabla v|^2 \, .
 \end{align}
   Applying this with $v= \dv_f \, Y$ and $\mu = \frac{1}{2}$ and integrating over $M$, we conclude that
 \begin{align}
 	\| \Hess_{  \dv_f \, Y } \|_{L^2} = 0 \, .
\end{align}
 It follows that $\nabla \dv_f \, Y$ is a non-trivial parallel vector field, giving  the desired splitting.
\end{proof}

There is a second   distinction between these two types of Killing fields.  The translations are generated by gradient vector fields coming from least eigenfunctions on the shrinking soliton.  On the other hand, the Killing fields that preserve $f$ turn out to be orthogonal to all gradient vector fields.  Both types of vector fields   satisfy  eigenvalue equations for the drift Laplacian $\cL$, but at different eigenvalues.

\subsection{Geometric estimates for shrinkers}
We will next recall several useful formulas for shrinkers and some geometric estimates.  First, 
taking the trace of the shrinker equation gives that
 \begin{align}
\Delta\,f+S&=\frac{n}{2}  \, , \label{ee:aronson1} 
\end{align}
where $S$ is the scalar curvature.
On a complete shrinker, it is well-known (\cite{H}, cf. \cite{ChLN}) that $f$ can be normalized so that
 \begin{align}
|\nabla f|^2+S&=f\, . \label{ee:aronson2}
\end{align}
Since  $S\geq 0$ by \cite{Cn} (see \cite{ChLY} for an improved bound on non-compact shrinkers), the function $b = 2\, \sqrt{f}$ is nonnegative and satisfies 
  \begin{align}	\label{e:gradb1}
  	|\nabla b| \leq 1 \, .
\end{align}
  
We will also need some geometric estimates by Cao-Zhou for shrinkers.  First,  Theorem $1.1$ in \cite{CaZd} gives $c_1, c_2$, depending only on   $B_1(x_0) \subset M$ so that 
 \begin{align}	\label{e:ccz}
 	\frac{1}{4} \left( r(x) - c_1 \right)^2 &\leq f (x) \leq  	\frac{1}{4} \left( r(x) + c_2 \right)^2 \, ,
 \end{align}
 where $r(x)$ is the distance to a fixed point $x_0$ (the constants can be made universal if $x_0$ is chosen at the minimum of $f$).    Second, Theorem $1.2$ in \cite{CaZd} gives that shrinkers have at most Euclidean volume growth: There exists 
 $c_3$ so that
 \begin{align}	\label{e:cczV}
 	\Vol (B_r (x_0)) \leq c_3 \, r^n \, .
\end{align}

 \section{Small eigenvalues and almost Killing fields}
 
 Throughout this section, $(M,g,f)$ is a complete non-compact gradient shrinking Ricci soliton (i.e., $\kappa = \frac{1}{2}$).  
 
 \vskip2mm
 The operator $\cP$ was constructed to vanish on Killing fields.
 The next lemma uses the variational characterizations of eigenvalues and the exponential decay of the weight to find low eigenvalues for $\cP$ when there is an approximate symmetry on a large ball.
 
 \begin{Lem}	\label{l:exist}
 Suppose that there is a (non-trivial) compactly supported vector field $V$  with 
 \begin{align}	\label{e:projn}
 	 \| \dv_f^* V  \|^2  \leq \bar{\mu}   \|V \|_{L^2}^2 \, ,
 \end{align}
 where $\bar{\mu} < 1$.
 Then there is a $W^{1,2}$ vector field $Z$ with $\| Z \|_{L^2} = 1$,  so that
 \begin{enumerate}
 \item[(A)] $\cP \, Z = \mu \, Z$, with $\mu \in [0,\bar{\mu}]$, and $\| \dv_f^* Z \|_{L^2}^2 = \mu$.
 \item[(B)]  $\dv_f \, Z \in W^{1,2}$ satisfies $\cL \, \dv_f  Z = -\left( \frac{1}{2} + \mu \right) \, \dv_f Z$ and $\| \dv_f \, Z \|_{L^2}^2 \leq 4\, \mu + 1$.
 \end{enumerate}
 \end{Lem}
 
 \begin{proof}
 The operator $\cP$ is self-adjoint by construction.
 Lemma $4.20$ in \cite{CM4} gives that $\cP$ has a complete basis of smooth $W^{1,2}$ eigen-vector fields $Y_i$ with eigenvalues 
 \begin{align}
 	\mu_i \to \infty \, , 
\end{align}
and with  $\| Y_i \|_{L^2} = 1$.  Moreover, Lemma $4.20$ in \cite{CM4} also gives that
 \begin{align}
 	  \int | \dv_f^* \, Y_i|^2 \, \e^{-f} = \int \langle Y_i , \cP \, Y_i \rangle \, \e^{-f} = \mu_i \, .
\end{align}
Expanding $V$ by projecting onto the $Y_i$'s, we write $V$ as
\begin{align}
	V = \sum_i a_i \, Y_i \, , 
\end{align}
where each $a_i \in \RR$ and 
\begin{align}	\label{e:fourier}
	\sum_i  a_i^2 = \| V \|_{L^2}^2  \, .
\end{align}
 Since the $Y_i$'s are $L^2$-orthonormal, \eqr{e:projn} gives that
\begin{align}
	\bar{\mu} \, \| V \|_{L^2}^2 &\geq \| \dv_f^* V \|_{L^2} =  \int \langle V , \cP \, V \rangle \, \e^{-f} = \sum_i a_i^2 \, \mu_i \notag \\
	&\geq \mu_1 \, \sum_i a_i^2 \, .
\end{align}
Comparing this with  \eqr{e:fourier}, we see that $\mu_1$ (the smallest $\mu_i$) is at most $\bar{\mu}$.  Set $Z= Y_1$.

Since $Z , \cP \, Z \in L^2$, Lemma \ref{l:interpcP} and Proposition \ref{p:4main0} give  that
\begin{align}
	\dv_f \, Z \in W^{1,2} {\text{ and }} \nabla Z \in L^2 \, .
\end{align}
  The $L^2$ bound on $\dv_f Z$ in (B) follows from Lemma \ref{l:interpcP}.
Finally, \eqr{e:claim1} gives that $\dv_f Z$ satisfies the eigenfunction equation
$\cL \, \dv_f Z = -\left( \frac{1}{2} + \mu \right) \, \dv_f \, Z$.
 \end{proof}

  The weighted $L^2$ bound on $\dv_f^*Z$ forces it to be small in the region where $f$ is small, but it says almost nothing where $f$ is large and the weight $\e^{-f}$ is very small.  To get better bounds when $f$ is large, we instead rely on polynomial growth estimates developed in \cite{CM4}.  To explain this, 
   given a tensor $w$ define the ``weighted spherical average'' $I_w(r)$ by
 \begin{align}
 	I_w (r) = r^{1-n} \, \int_{b=r} |w|^2 \, |\nabla b| \, .
 \end{align}
 A priori, this is well-defined at regular values of $b$, but  Lemma $3.27$ in \cite{CM4} shows that $I_w(r)$ can be extended to all $r$, this extension
  is differentiable almost everywhere, and is absolutely continuous as a function of $r$; cf. \cite{B}, \cite{CM1}, \cite{CM3}.  
  
  We have the following polynomial growth bounds:
  
  \begin{Pro}	\label{p:pgCM4}
 \cite{CM4}.  Given $\bar{\lambda}$, there exists $r_0$  so that if $w$ is an $L^2$ tensor with 
 \begin{align}	\label{e:LBL}
 	\langle \cL w , w \rangle \geq - \bar{\lambda} \, |w|^2 \, , 
\end{align}
then for any $r_2 > r_1 \geq r_0$ we have that
\begin{align}
	I_w(r_2) \leq  2\, \left( \frac{r_2}{r_1} \right)^{5\, \bar{\lambda}} \, I_w (r_1) \, .
\end{align}
\end{Pro}
  
  Proposition \ref{p:pgCM4} is a special case of Theorem $3.4$ in \cite{CM4}.   This proposition requires  the lower bound \eqr{e:LBL}
  for $\langle \cL w , w \rangle$, which would hold if $w$ satisfied an eigenvalue equation for $\cL$.  However,   we will want to apply it to a vector field that satisfies an eigenvalue equation for $\cP$.   The next result gives a decomposition for that equation that makes this possible (this a special case of Proposition $4.6$ in \cite{CM4}):

 \begin{Pro}  \label{p:4main0}
 \cite{CM4}.
 If    $\cP\,Y=\mu\,Y$  and we set $Z=Y+\frac{2}{2\mu+1}\,\nabla\,\dv_f\,(Y)$, then 
 $\dv_f\,(Z)=0$ and 
 \begin{align}
 ( \cL + \mu) \,\nabla\,\dv_f\,(Y) &=0 \,  , \label{e:e1}\\
\left( \cL +2\,\mu+\frac{1}{2} \right)\,Z&= 0 \, .\label{e:eigZ}
 \end{align}
 Moreover, if $Y  \in L^2$, then   $\|Y\|^2=\|Z\|^2+\left(\mu+\frac{1}{2} \right)^{-2}\,\|\nabla\,\dv_f\,(Y)\|^2$.
     \end{Pro}

  \vskip2mm
We are now prepared to prove that $\dv_f^* Y$ grows at most polynomially when $Y$ is an eigenvector field for $\cP$.
 
\begin{Thm}	\label{t:main1}
Suppose that $M,g$ has bounded curvature $|\R| \leq C_1$ and  $Y$ is a $W^{1,2}$ vector field with $\| Y \|_{L^2}=1$ 
that satisfies (A) and (B).  There exist  $C_2 , R$ so that  for all $r \geq R$
\begin{align}
	I_{\dv_f^* Y} (r) \leq C_1 \, r^{C_2} \, \mu \, .
\end{align}
\end{Thm}

\begin{proof}
Since $\cP \, Y = \mu \, Y$, Proposition \ref{p:4main0} gives that $Z= Y + \frac{2}{1+2\, \mu} \, \nabla \dv_f \, Y$ satisfies
\begin{align}	\label{e:setZ}
	\cL \, Z &= - \left( 2\, \mu + \frac{1}{2} \right) \, Z \, , \\
	\cL \, \nabla \, \dv_f \, Y &= - \mu \, \nabla \, \dv_f \, Y \, .  \label{e:setnd}
\end{align}
Applying \eqr{e:firstc} to \eqr{e:setZ} and \eqr{e:setnd} gives
\begin{align}
	L \, \dv_f^* \, Z &= \dv_f^* \left(\cL + \frac{1}{2} \right) \, Z = 2\, \mu \, \dv_f^*   \, Z \, , \label{e:combo1} \\
	L \, \dv_f^* \nabla \dv_f \, Y & = \dv_f^* \left(\cL + \frac{1}{2} \right) \, \nabla \, \dv_f \, Y = \left( \frac{1}{2} - \mu \right) \,  \dv_f^* \, \nabla {\dv_f Y} \, .
\end{align}
The last equation can be rewritten as
\begin{align}
	L \, \Hess_{\dv_f Y} =  \left( \frac{1}{2} - \mu \right) \, \Hess_{\dv_f Y}  \, . \label{e:combo2}
\end{align}
Since $M,g$ has bounded curvature and $\mu$ is also bounded, \eqr{e:combo1} and \eqr{e:combo2} give $C$ so that
\begin{align}	\label{e:combo3}
	\langle \cL \, w , w \rangle \geq - C \, |w|^2 {\text{ for }} w =  \Hess_{ \dv_f Y} {\text{ or }} w = \dv_f^* \, Z \, .
\end{align}
Thus, Proposition \ref{p:pgCM4} applies, giving $R$ and $C_1$ so that   $$w =  \Hess_{ \dv_f Y} {\text{ or }}  \dv_f^* \, Z $$ both satisfy
\begin{align}	\label{e:conclude}
	I_w (r_2) \leq \left( \frac{r_2}{r_1} \right)^{ C_1} \, I_w (r_1) {\text{ for any }} r_2 > r_1 \geq R \, .
\end{align}
Next, observe that (B) and the drift Bochner formula give that
\begin{align}
	\| \Hess_{\dv_f Y} \|_{L^2}^2 \leq \mu \, .
\end{align}
On the other hand, (A) gives that $\| \dv_f^* Y \|_{L^2}^2 \leq \mu$ as well.  It follows that 
\begin{align}
	\| \dv_f^* Z \|_{L^2}^2 \leq C \, \mu \, . 
\end{align}
Since $|\nabla b| \leq 1$, the co-area formula gives some $r_1 \geq R$ and a constant $C_2$ so that
\begin{align}
	I_w(r_1) \leq C_2 \, \mu {\text{ for }} w =  \Hess_{ \dv_f Y} {\text{ or }} w = \dv_f^* \, Z \, .
\end{align}
Using this in \eqr{e:conclude} and then writing $\dv_f^*Y$ as a combination gives the theorem.
\end{proof}

We are now prepared to prove the main theorem.

\begin{proof}[Proof of Theorem \ref{t:main}]
The first step is to cutoff  the vector field to get a compactly supported vector field $V$ that we can use in Lemma \ref{l:exist}.  To do this, define a cutoff function $\eta$ with $0 \leq \eta \leq 1$ that has support in 
$f < \frac{r^2}{4}$, cuts off in distance $r^{-1}$, has $|\nabla \eta | \leq 2 \, r$  and so that
\begin{align}
	\e^{-f} \leq \e^{-1} \, \e^{ - \frac{r^2}{4}} {\text{ on the support of }} |\nabla \eta|  
	\, .
\end{align}
The last bound uses that $|\nabla b| \leq 1$ by \eqr{e:gradb1}.  Now set $V = \eta \, Y$.  Since $|Y| \leq C_1 \, r$ on the support of $\eta$, it follows that
\begin{align}	\label{e:Vbig}
	\| V \|_{L^2}^2 &\geq \int_{\eta =1} |Y|^2 \, \e^{-f} = 1 - \int_{0<\eta < 1} |Y|^2 \, \e^{-f} \geq 1 - C  \, r^{2+n} \, \e^{ - \frac{r^2}{4}} \, ,
\end{align}
where the last bound also used the volume bound \eqr{e:cczV}.  Similarly, we have that
\begin{align}	 
	\| \dv_f^* V \|_{L^2}^2 &\leq 2 \, \int_{f < r^2/4} |\dv_f^* Y|^2 \, \e^{-f} + C \, r^{4+n} \, \e^{ -\frac{r^2}{4}} \notag \\
	&\leq 2\, \bar{\mu} +   C \, r^{4+n} \, \e^{ -\frac{r^2}{4}}  \, .
\end{align}
After multiplying $V$ by a constant so that the $L^2$ norm is one, we have that
\begin{align}	\label{e:Liesmall}
	\| \dv_f^* V \|_{L^2}^2  
	\leq 3\, \bar{\mu} +   C \, r^{4+n} \, \e^{ -\frac{r^2}{4}} < 1 \, .
\end{align}
Lemma \ref{l:exist}
gives a $W^{1,2}$ vector field $Z$ with $\| Z \|_{L^2} = 1$  and satisfying (A) and (B).  In particular, $\cP \, Z = \mu \, Z$ with
\begin{align}
	\mu \leq 3\, \bar{\mu} +   C \, r^{4+n} \, \e^{ -\frac{r^2}{4}} < 1 \, .
\end{align}
 This gives (Z1). The fine growth bound (Z2) follows from Theorem \ref{t:main1}.
\end{proof}

\end{document}